\documentclass[11pt]{amsart}

\usepackage{amsmath,amssymb,mathtools}
\usepackage{enumitem}
\usepackage{microtype}
\usepackage[colorlinks=true,linkcolor=blue,citecolor=blue,urlcolor=blue,
  pdftitle={Complexity-Sensitive Additive Energy and Off-Diagonal Young Inequalities on Bounded-Degree Algebraic Varieties},
  pdfauthor={Xiyu Hu},
  pdfkeywords={additive energy, algebraic varieties, polynomial partitioning, Young inequality, difference varieties, convex curves}]{hyperref}
\usepackage[nameinlink,noabbrev]{cleveref}

\numberwithin{equation}{section}

\newtheorem{theorem}{Theorem}[section]
\newtheorem{proposition}[theorem]{Proposition}
\newtheorem{lemma}[theorem]{Lemma}
\newtheorem{corollary}[theorem]{Corollary}

\theoremstyle{definition}

\theoremstyle{remark}
\newtheorem{remark}[theorem]{Remark}

\newcommand{\R}{\mathbb R}
\newcommand{\C}{\mathbb C}
\newcommand{\Z}{\mathbb Z}
\newcommand{\T}{\mathbb T}

\newcommand{\cB}{\mathcal B}
\newcommand{\cD}{\mathcal D}
\newcommand{\cF}{\mathcal F}
\newcommand{\cJ}{\mathcal J}

\newcommand{\cR}{\mathcal R}
\newcommand{\cS}{\mathcal S}

\newcommand{\supp}{\operatorname{supp}}

\newcommand{\Stab}{\operatorname{Stab}}
\newcommand{\Bad}{\operatorname{Bad}}
\newcommand{\bzero}{b_0}
\newcommand{\eps}{\varepsilon}
\newcommand{\1}{\mathbf 1}
\newcommand{\wt}[1]{\widetilde{#1}}
\newcommand{\card}[1]{\lvert #1\rvert}
\newcommand{\norm}[2][]{\lVert #2\rVert_{#1}}

\title[Additive energy on algebraic varieties]{Complexity-Sensitive Additive Energy and Off-Diagonal Young Inequalities on Bounded-Degree Algebraic Varieties}
\author{Xiyu Hu}
\address{School of Mathematical Sciences, University of Chinese Academy of Sciences}

\email{hxypqr@gmail.com}
\date{Working draft, August 19, 2026}


\begin{document}

\begin{abstract}
We develop additive-energy estimates and weighted Young inequalities for finite sets on bounded-degree real algebraic varieties.  For an irreducible $m$-dimensional variety $V$, let
\[
 \sigma(V)=2m-\dim\overline{V-V}^{\,\mathrm{Zar}},
 \qquad
 \alpha(V)=\max\left\{2,1+\frac{2\sigma(V)}m\right\}.
\]
For every $a\in[\alpha(V),3)$ we define a finite-degree translation--partition flag parameter $\Lambda_{a,R}(X;V)$ and prove
\[
 E(X)\ll \Lambda_{a,R}(X;V)^{3-a}|X|^{a+\varepsilon}.
\]
This recovers the line-concentration theorem of Jing and Wu for algebraic surfaces in $\mathbb R^3$.  For codimension-two quadratic threefolds
\[
 \{(u,Q_1(u),Q_2(u)):u\in\mathbb R^3\}\subset\mathbb R^5
\]
with positive-definite $Q_1$ and simple generalized spectrum, we prove the sharp estimate $E(X)\ll_\varepsilon |X|^{2+\varepsilon}$ without a flag loss.  Hereditary versions of these estimates imply weighted $L^4$ restriction bounds and off-diagonal Young inequalities; at the near-diagonal threshold the sharp region is
\[
 1\le p,q\le2,\qquad p^{-1}+q^{-1}\ge1.
\]
We also prove a sharp turning-complexity extension of the Cushman--Demeter--Wu theorem: $J_3(P)\ll_\varepsilon \kappa(P)^2|P|^{3+\varepsilon}$, with matching examples at every power scale.
\end{abstract}

\maketitle

\section{Introduction}

\subsection{Additive energy and geometric support restrictions}
Let $G$ be a torsion-free abelian group and let $X\subset G$ be finite.  The additive energy of $X$ is
\begin{equation}\label{eq:intro-energy}
 E(X)=\sum_{t\in G}r_{X-X}(t)^2,
 \qquad
 r_{X-X}(t)=\#\{(x,y)\in X^2:x-y=t\}.
\end{equation}
It satisfies the universal bounds
\[
 2\card X^2-\card X\leq E(X)\leq \card X^3.
\]
The lower bound is the diagonal scale; the upper bound is attained, up to constants, by arithmetic progressions.  If $G_0$ denotes the subgroup generated by $X$, then $G_0\cong\Z^r$ for some $r$, and Plancherel gives
\begin{equation}\label{eq:intro-fourier}
 E(X)=\int_{\T^r}\left|\widehat{\1_X}(\xi)\right|^4\,d\xi.
\end{equation}
Thus a near-diagonal energy theorem is simultaneously a discrete $L^4$ restriction theorem.

Curvature and algebraic complexity can suppress additive multiplicity.  Bourgain and Demeter related such questions to decoupling and discrete restriction \cite{BourgainDemeter}.  Jing and Wu recently proved that if $F:\R^3\to\R$ is irreducible of degree at least two, then every finite $X\subset Z(F)$ satisfies
\begin{equation}\label{eq:jw-intro}
 E(X)\ll_{\deg V,\eps}\Lambda_V(X)\card X^{2+\eps},
 \qquad
 \Lambda_F(X)=\max\left\{1,\sup_{\ell\subset Z(F)}\card{X\cap\ell}\right\},
\end{equation}
where the supremum is over affine lines contained in the surface \cite{JingWu}.  Their theorem resolves the near-diagonal problem for arbitrary algebraic surfaces and identifies concentration on surface-contained lines as the only obstruction.

There is a related higher-order phenomenon on curves.  For a finite $P\subset\R^2$, define
\begin{equation}\label{eq:J3-intro}
 J_3(P)=\#\{(p_1,\dots,p_6)\in P^6:p_1+p_2+p_3=p_4+p_5+p_6\}.
\end{equation}
Cushman, Demeter, and Wu proved that $J_3(P)\ll_\eps \card P^{3+\eps}$ whenever $P$ lies on a strictly convex graph \cite{CDW}.  Their proof combines an interlacing structure inside equal-sum fibers with Fourier orthogonality for ordered chord cones.

The present paper has two goals.  The first is to identify the dimension and translation-complexity parameters that govern the polynomial-partitioning argument on general algebraic varieties.  The second is to convert the resulting set estimates into weighted and off-diagonal Young inequalities, in a form that records the sharp critical exponents.  This functional viewpoint is in the tradition of the sharp Young and Fourier inequalities of Beckner and Brascamp--Lieb \cite{Beckner,BrascampLieb}, while the geometric decomposition follows the polynomial-partitioning philosophy initiated by Guth and Katz \cite{GuthKatz} and developed over general varieties by Walsh.

\subsection{Difference varieties and the intrinsic exponent}
Let $V\subset\R^n$ be an irreducible real algebraic variety whose complexification $V_{\C}$ is irreducible and whose real locus is Zariski dense.  Put
\[
 m=\dim V_{\C},
 \qquad
 \cD(V)=\overline{V_{\C}-V_{\C}}^{\,\mathrm{Zar}}\subset\C^n,
 \qquad
 d(V)=\dim\cD(V).
\]
The difference map
\[
 \delta_V:V_{\C}\times V_{\C}\longrightarrow \cD(V),
 \qquad
 \delta_V(x,y)=x-y,
\]
has generic fiber dimension
\begin{equation}\label{eq:sigma-intro}
 \sigma(V)=2m-d(V).
\end{equation}
Accordingly, for a generic difference $t$, the translate intersection $V_{\C}\cap(V_{\C}+t)$ has dimension $\sigma(V)$.  The exponent produced by two applications of the crossing estimate is
\begin{equation}\label{eq:alpha-intro}
 \alpha(V)=\max\left\{2,1+\frac{2\sigma(V)}m\right\}.
\end{equation}
The diagonal contribution forces the first term.  The second term is the fixed point of the polynomial-partition recurrence.

If $V^m\subset\R^{m+c}$ has maximal difference dimension, then
\[
 d(V)=\min(2m,m+c),
 \qquad
 \sigma(V)=\max(0,m-c),
\]
and hence
\begin{equation}\label{eq:mc-intro}
 \alpha_{m,c}=\max\left\{2,3-\frac{2c}{m}\right\}.
\end{equation}
The line $m=2c$ is the critical boundary of the near-diagonal range.  The surface case $V^2\subset\R^3$ and the codimension-two cases $V^3\subset\R^5$ and $V^4\subset\R^6$ lie on or below this threshold.

\subsection{The flagged energy theorem}
A general variety can contain affine rulings, positive-dimensional translation stabilizers, and lower-dimensional subvarieties with a worse intrinsic exponent.  These phenomena prevent a uniform parameter-free theorem.  We therefore introduce, in \cref{sec:flag-definitions}, a finite-degree translation--partition flag parameter $\Lambda_{a,R}(X;V)$.  It records:
\begin{enumerate}[label=(\roman*)]
 \item concentration on cosets of translation stabilizers;
 \item concentration on translate intersections whose dimension exceeds the generic value;
 \item concentration on bounded-degree flag subvarieties whose intrinsic exponent is larger than the target exponent.
\end{enumerate}
The parameter is monotone under passage to subsets and satisfies $1\leq\Lambda_{a,R}(X;V)\leq\card X$.

Our main general theorem is the following.  The algebraic conventions and the precise definition of $\Lambda_{a,R}$ are given in \cref{sec:algebraic-preliminaries,sec:flag-definitions}.

\begin{theorem}\label{thm:flagged-main}
Let $n,\Delta\geq1$, let $2\leq a<3$, and let $\eps>0$.  There are constants
\[
 R=R(n,\Delta,a,\eps),
 \qquad
 C=C(n,\Delta,a,\eps)
\]
with the following property.  Let $V\subset\R^n$ be an admissible irreducible algebraic variety of degree at most $\Delta$ such that $\alpha(V)\leq a$.  Then every finite $X\subset V$ satisfies
\begin{equation}\label{eq:flagged-main-intro}
 E(X)\leq C\,\Lambda_{a,R}(X;V)^{3-a}\card X^{a+\eps}.
\end{equation}
\end{theorem}

Taking $a=\alpha(V)$ gives the intrinsic form whenever $\alpha(V)<3$.  The endpoint $\alpha(V)=3$ is governed by the universal trivial estimate $E(X)\leq|X|^3$ and has no contractive partition recurrence.  When $V$ is an irreducible nonplanar surface in $\R^3$, the flag parameter is comparable, with degree-dependent constants, to $\Lambda_F(X)$ in \eqref{eq:jw-intro}; see \cref{cor:recover-jw}.  Thus \cref{thm:flagged-main} recovers the theorem of Jing and Wu.  Its proof follows their two-crossing mechanism, but runs on a spatial flag of partition walls and uses the difference variety to determine the correct fiber dimension.

The power $3-a$ has the expected interpolation form.  If a base example has energy of order $N^a$ and one takes its Cartesian product with an arithmetic progression of length $L$, then the resulting energy is of order $L^{3-a}N^a$.  In particular, the power is sharp in the classes for which the base exponent is attained.

\subsection{A sharp codimension-two threefold theorem}
The flag parameter in \cref{thm:flagged-main} is deliberately safe.  It can overcount exceptional translate fibers that have large individual occupancy but small total energy.  The first nontrivial example where this distinction matters is a three-dimensional quadratic graph in $\R^5$.

Let $Q_j(u)=u^{\mathsf T}A_ju$ be quadratic forms on $\R^3$.  We say that $(Q_1,Q_2)$ has simple generalized spectrum if $A_1$ is positive definite and $A_1^{-1/2}A_2A_1^{-1/2}$ has three distinct eigenvalues.

\begin{theorem}\label{thm:quadratic-threefold-intro}
Let
\[
 \Sigma=\{(u,Q_1(u),Q_2(u)):u\in\R^3\}\subset\R^5,
\]
where $(Q_1,Q_2)$ has simple generalized spectrum.  Then, for every $\eps>0$ and every finite $X\subset\Sigma$,
\begin{equation}\label{eq:quadratic-threefold-intro}
 E(X)\ll_\eps \card X^{2+\eps}.
\end{equation}
The exponent $2$ is sharp.
\end{theorem}

After linear changes of variables, the model becomes
\[
 \Phi(u)=\bigl(u,|u|^2,\lambda_1u_1^2+\lambda_2u_2^2+\lambda_3u_3^2\bigr),
 \qquad \lambda_i\neq\lambda_j.
\]
For a generic difference, the parameter fiber is an affine line.  It becomes a plane only when the parameter displacement lies in one of the three eigendirections.  A Gram-matrix argument shows that the total energy of all these rank-drop directions is $O(N^2)$.  Polynomial partitioning in the parameter space produces $O(D^3)$ cells, while the two crossing estimates cost only $D^2$; at exponent $2$ the cellular recursion therefore gains a genuine factor $D^{-1}$.  The two-dimensional wall is reduced to the Jing--Wu theorem by a finite Freiman projection to $\R^3$.

The simple-spectrum hypothesis cannot be removed without replacing it by a geometric nondegeneracy condition.  If $Q_2$ is proportional to $Q_1$, the graph lies in an affine copy of a three-dimensional paraboloid in $\R^4$, where the universal threshold is $N^{7/3}$ rather than $N^2$.

\subsection{Weighted and off-diagonal Young inequalities}
A hereditary set-energy theorem has a functional consequence.  If every $A\subset X$ satisfies
\[
 E(A)\ll K^{3-a}\card A^{a+o(1)},
\]
then
\begin{equation}\label{eq:weighted-intro}
 \norm[L^4]{\widehat f}
 \ll K^{(3-a)/4}\card X^{o(1)}\norm[\ell^{4/a}]{f},
 \qquad \supp f\subset X.
\end{equation}
Combining two such estimates with Plancherel gives a geometric Young endpoint.  Interpolation with the classical endpoints $\ell^1*\ell^2\to\ell^2$ and $\ell^2*\ell^1\to\ell^2$ yields a full off-diagonal family.

The near-diagonal case $a=2$ is especially clean.

\begin{theorem}\label{thm:offdiag-alpha2-intro}
Suppose finite sets $X,Y$ satisfy hereditary estimates
\[
 E(A)\ll_\eta K_X\card A^{2+\eta}
 \quad(A\subset X),
 \qquad
 E(B)\ll_\eta K_Y\card B^{2+\eta}
 \quad(B\subset Y)
\]
for every $\eta>0$.  Let $1\leq p,q\leq2$ and assume
\[
 \frac1p+\frac1q\geq1.
\]
Put
\[
 \theta=\max\left\{0,3-2\left(\frac1p+\frac1q\right)\right\}.
\]
Then for every $\eps>0$ and all $f,g$ supported on $X,Y$,
\begin{equation}\label{eq:offdiag-alpha2-intro}
 \norm[\ell^2]{f*g}
 \ll_{p,q,\eps}
 (K_XK_Y)^{\theta/4}(\card X\card Y)^\eps
 \norm[\ell^p]{f}\norm[\ell^q]{g}.
\end{equation}
The exponent region is sharp up to the subpolynomial loss.  If the class contains arithmetic progressions on affine rulings, the joint power of $K_XK_Y$ is also sharp.
\end{theorem}

For the threefold in \cref{thm:quadratic-threefold-intro}, one has $K_X=K_Y=1$.  For algebraic surfaces in $\R^3$, one may take $K_X=\Lambda_F(X)$ and $K_Y=\Lambda_G(Y)$.  Thus the same off-diagonal theorem simultaneously applies to the recent results of Jing--Wu and to the new codimension-two case.

\subsection{A sharp turning-complexity theorem for three-fold curve energy}
Our final main result is a parallel complexity theorem for \eqref{eq:J3-intro}.  Fix a linear functional $\ell$ that is injective on $P\subset\R^2$ and order
\[
 \ell(p_1)<\cdots<\ell(p_N).
\]
Choose a complementary linear functional $m$ and form the adjacent slopes
\[
 s_i=\frac{m(p_{i+1})-m(p_i)}{\ell(p_{i+1})-\ell(p_i)}.
\]
The least number of consecutive point blocks on each of which the corresponding slope sequence is strictly increasing or strictly decreasing is denoted by $\kappa_\ell(P)$.  This number does not depend on the choice of $m$.

\begin{theorem}\label{thm:turning-intro}
For every $\eps>0$,
\begin{equation}\label{eq:turning-intro}
 J_3(P)\ll_\eps \kappa_\ell(P)^2\card P^{3+\eps}.
\end{equation}
Moreover, for every pair of integers $K,n\geq1$, there is a set $P$ with $\card P=Kn$ for which
\[
 \kappa_\ell(P)\asymp K,
 \qquad
 J_3(P)\asymp K^2\card P^3.
\]
Thus the power $\kappa^2$ is sharp at every intermediate power scale.
\end{theorem}

The proof partitions into signed-convex leaves and uses one global $L^6$ strip projection estimate.  The strips charge the leaves only once; a repeated node-by-node triangle inequality would lose an additional factor of $\kappa$.  Each leaf is interpolated by a strictly convex or strictly concave graph and is handled by the theorem of Cushman--Demeter--Wu.

\subsection{Organization}
Basic Fourier and algebraic tools are collected in \cref{sec:preliminaries}.  The flagged energy theorem is proved in \cref{sec:flagged-energy}.  The codimension-two quadratic threefold is treated in \cref{sec:quadratic-threefold}.  Weighted and off-diagonal inequalities occupy \cref{sec:weighted-young}.  The turning-complexity theorem is proved in \cref{sec:turning}.  Sharpness, higher-dimensional consequences, and the remaining geometric difficulties are discussed in \cref{sec:higher-dimensional}.

\section{Preliminaries}\label{sec:preliminaries}

\subsection{Fourier identities on finite supports}
Every finite subset of $\R^n$ is contained in a finitely generated torsion-free subgroup, hence in a group isomorphic to $\Z^r$.  We freely identify such a subgroup with $\Z^r$ when using Fourier analysis.  This does not change any additive relation.

For a finitely supported function $f$ on an abelian group, put
\[
 \wt f(x)=\overline{f(-x)},
 \qquad
 R_f(t)=(f*\wt f)(t).
\]
For finite sets $A,B$, write
\[
 r_{A-B}(t)=\#\{(a,b)\in A\times B:a-b=t\},
\]
and define the mixed energy
\begin{equation}\label{eq:mixed-energy}
 E(A,B)=\sum_t r_{A-B}(t)^2.
\end{equation}
Thus $E(A)=E(A,A)$.

\begin{lemma}\label{lem:energy-identities}
Let $f,g$ be finitely supported.  Then
\begin{align}
 \norm[\ell^2]{f*g}^2
 &=\sum_t R_f(t)\overline{R_g(t)}
 =\int \lvert\widehat f\rvert^2\lvert\widehat g\rvert^2,\label{eq:bilinear-correlation}\\
 \norm[L^4]{\widehat f}^4
 &=\sum_t\lvert R_f(t)\rvert^2.\label{eq:weighted-energy}
\end{align}
In particular,
\begin{equation}\label{eq:mixed-identity}
 E(A,B)=\sum_t r_{A-A}(t)r_{B-B}(t).
\end{equation}
If $A=\bigsqcup_{j=1}^J A_j$, then
\begin{equation}\label{eq:union-fourth-root}
 E(A)^{1/4}\leq\sum_{j=1}^J E(A_j)^{1/4}.
\end{equation}
\end{lemma}

\begin{proof}
Identify the generated subgroup with $\Z^r$ and use Fourier series on $\T^r$.  Equations \eqref{eq:bilinear-correlation} and \eqref{eq:weighted-energy} follow from Plancherel.  Equation \eqref{eq:mixed-identity} is the indicator-function case of \eqref{eq:bilinear-correlation}.  Finally,
\[
 E(A)^{1/4}=\norm[L^4]{\widehat{\1_A}}
 \leq\sum_j\norm[L^4]{\widehat{\1_{A_j}}}
 =\sum_j E(A_j)^{1/4}.
\]
\end{proof}

We shall also use the elementary bounds
\begin{equation}\label{eq:energy-universal}
 2\card A^2-\card A\leq E(A)\leq\card A^3.
\end{equation}

\subsection{Algebraic conventions}\label{sec:algebraic-preliminaries}
An \emph{admissible irreducible real algebraic variety} $W\subset\R^n$ means the real locus of an irreducible complex affine variety $W_{\C}\subset\C^n$ defined over $\R$, with $W$ Zariski dense in $W_{\C}$.  We set
\[
 \dim W=\dim_{\C}W_{\C},
 \qquad
 \deg W=\deg W_{\C}.
\]
General bounded-degree real algebraic sets admit a bounded real-algebraic stratification: after separating conjugate complex components and their lower-dimensional real intersections, the real locus is covered by $O_{n,\Delta}(1)$ admissible irreducible pieces of bounded degree, together with lower-dimensional pieces; see, for example, \cite{BPR}.  Since \eqref{eq:union-fourth-root} handles finite unions, it is enough to formulate the main argument for admissible irreducible varieties.

For such a $W$ of dimension $k\geq1$, define
\begin{equation}\label{eq:difference-variety}
 \cD(W)=\overline{W_{\C}-W_{\C}}^{\,\mathrm{Zar}},
 \qquad
 d(W)=\dim\cD(W),
 \qquad
 \sigma(W)=2k-d(W).
\end{equation}
The fiber-dimension theorem implies that $\sigma(W)$ is the generic complex dimension of
\[
 W_{\C}\cap(W_{\C}+t),
 \qquad t\in\cD(W).
\]
We put
\begin{equation}\label{eq:alpha-W}
 \alpha(W)=\max\left\{2,1+\frac{2\sigma(W)}k\right\}.
\end{equation}

The real translation stabilizer is
\begin{equation}\label{eq:stabilizer}
 H_W=\Stab(W)=\{t\in\R^n:W+t=W\}.
\end{equation}

\begin{lemma}\label{lem:stabilizer-linear}
The set $H_W$ is a real linear subspace of $\R^n$.
\end{lemma}

\begin{proof}
It is plainly an additive subgroup.  If $t\in H_W$, then $W+jt=W$ for every $j\in\Z$.  For every polynomial $P$ vanishing on $W_{\C}$ and every $x\in W_{\C}$, the polynomial $s\mapsto P(x+st)$ vanishes at all integers, hence identically.  Thus $W_{\C}+st=W_{\C}$ for every $s\in\C$, and in particular $st\in H_W$ for every real $s$.  Therefore $H_W$ is a real vector subspace.
\end{proof}

Define the exceptional translation set
\begin{equation}\label{eq:bad-set}
 \cB(W)=\left\{t\in\R^n\setminus H_W:
 \dim_{\C}\bigl(W_{\C}\cap(W_{\C}+t)\bigr)>\sigma(W)
 \right\}.
\end{equation}
For a finite $A\subset W$, set
\begin{equation}\label{eq:lambda-W}
 \lambda_W(A)=\max\left\{
 1,
 \sup_{z\in\R^n}\card{A\cap(z+H_W)},
 \sup_{t\in\cB(W)}\card{A\cap W\cap(W+t)}
 \right\}.
\end{equation}
Empty suprema are interpreted as zero.

\subsection{Polynomial partitioning and real component bounds}
We use two standard consequences of the polynomial method over varieties.  The first combines polynomial partitioning on an arbitrary variety with the usual component bound; see Walsh \cite{Walsh} and, in codimension at most two, Basu--Sombra \cite{BasuSombra}.

\begin{theorem}[Polynomial partitioning on a variety]\label{thm:partitioning-variety}
Let $W\subset\R^n$ be admissible irreducible of dimension $k$ and degree at most $\Delta$, and let $A\subset W$ be finite.  For every integer $D\geq2$, there is a polynomial $P$ of degree at most $D$, not vanishing identically on $W_{\C}$, such that every semialgebraically connected component of
\[
 W\setminus Z(P)
\]
contains at most
\[
 C_{n,\Delta}\frac{\card A}{D^k}
\]
points of $A$.  The number of components meeting $A$ is at most $C_{n,\Delta}D^k$.
\end{theorem}

The second is a form of the Barone--Basu sign-condition estimate \cite{BaroneBasu}.

\begin{theorem}[Component bound]\label{thm:component-bound}
Let $Z\subset\R^n$ be a real algebraic set of dimension at most $s$ and degree at most $B$.  Let $P_1,P_2$ be polynomials of degree at most $D$.  Then
\begin{equation}\label{eq:component-bound}
 \bzero\Bigl(Z\setminus\bigl(Z(P_1)\cup Z(P_2)\bigr)\Bigr)
 \leq C_{n,B}(D+1)^s.
\end{equation}
\end{theorem}

We also use standard affine Bezout bounds: intersections of bounded-degree varieties have bounded total degree, and if a line meets an algebraic variety of degree $B$ in more than $B$ points, then the line is contained in the variety.  These facts may be found, for example, in \cite{BPR}.

\section{The flagged energy theorem}\label{sec:flagged-energy}

\subsection{The flag parameter}\label{sec:flag-definitions}
Fix an admissible irreducible $V\subset\R^n$, a target exponent $2\leq a<3$, and a degree cutoff $R\geq\deg V$.  Let $\cF_R(V)$ denote the family of positive-dimensional admissible irreducible subvarieties $W\subset V$ with $\deg W\leq R$.  Define
\begin{equation}\label{eq:flag-parameter}
 \Lambda_{a,R}(X;V)
 =\max\left\{
 \begin{aligned}
 &1,\\
 &\sup_{\substack{W\in\cF_R(V)\\ \alpha(W)\leq a}}
 \lambda_W(X\cap W),\\
 &\sup_{\substack{W\in\cF_R(V)\\ \alpha(W)>a}}
 \card{X\cap W}
 \end{aligned}
 \right\}.
\end{equation}
The parameter is monotone: if $X'\subset X$, then
\begin{equation}\label{eq:flag-monotone}
 \Lambda_{a,R}(X';V)\leq\Lambda_{a,R}(X;V).
\end{equation}

The degree cutoff is essential.  Without it, one could fit a high-degree subvariety to a prescribed finite configuration.  In the proof of \cref{thm:flagged-main}, $R$ is chosen large enough to contain every descendant produced by the finite partition-wall induction.

\subsection{Exceptional translations}
For $A\subset W$, write
\[
 r_A(t)=r_{A-A}(t).
\]

\begin{lemma}\label{lem:exceptional-energy}
Let $W$ be admissible irreducible and let $A\subset W$ be finite.  Then
\begin{equation}\label{eq:exceptional-energy}
 \sum_{t\in H_W\cup\cB(W)}r_A(t)^2
 \leq 2\lambda_W(A)\card A^2.
\end{equation}
\end{lemma}

\begin{proof}
Partition $A$ into its intersections $A_L$ with the affine cosets of $H_W$.  If $t\in H_W$, then both points in a representation $x-y=t$ belong to the same coset.  Hence
\[
 \sum_{t\in H_W}r_A(t)
 =\sum_L\card{A_L}^2
 \leq\lambda_W(A)\card A.
\]
Since $r_A(t)\leq\card A$,
\begin{equation}\label{eq:stabilizer-energy}
 \sum_{t\in H_W}r_A(t)^2
 \leq\lambda_W(A)\card A^2.
\end{equation}

If $t\in\cB(W)$, then every first point in a representation of $t$ lies in $A\cap W\cap(W+t)$, so
\[
 r_A(t)\leq\lambda_W(A).
\]
Therefore
\begin{equation}\label{eq:excess-energy}
 \sum_{t\in\cB(W)}r_A(t)^2
 \leq\lambda_W(A)\sum_t r_A(t)
 =\lambda_W(A)\card A^2.
\end{equation}
Combining \eqref{eq:stabilizer-energy} and \eqref{eq:excess-energy} proves the result.
\end{proof}

The elementary inequality
\begin{equation}\label{eq:lambda-interpolation}
 \lambda M^2\leq\lambda^{3-a}M^a,
 \qquad 1\leq\lambda\leq M,
 \quad 2\leq a<3,
\end{equation}
will repeatedly convert \cref{lem:exceptional-energy} into the target scale.

\subsection{The two-crossing estimate}
Let $W\subset\R^n$ have dimension $k$ and put $\sigma=\sigma(W)$.  Apply \cref{thm:partitioning-variety} to $A\subset W$ with a polynomial $P$ of degree $D$.  Let
\[
 A_0=A\cap Z(P),
 \qquad
 Y=A\setminus A_0=\bigsqcup_i A_i,
\]
where the $A_i$ lie in distinct components of $W\setminus Z(P)$.

\begin{proposition}\label{prop:two-crossing}
With the notation above,
\begin{equation}\label{eq:two-crossing}
 E(Y)
 \leq C_{n,\deg W}\left(
 D^\sigma\lambda_W(A)\card A^2
 +D^{2\sigma}\sum_iE(A_i)
 \right).
\end{equation}
\end{proposition}

\begin{proof}
Let
\[
 \Bad(W)=H_W\cup\cB(W).
\]
Fix $t\notin\Bad(W)$.  The algebraic set
\[
 W_t=W\cap(W+t)
\]
has complex dimension at most $\sigma$ and degree bounded in terms of $\deg W$.  Apply \cref{thm:component-bound} to $W_t$ and the two polynomials
\[
 P(x),\qquad P(x-t).
\]
It follows that
\begin{equation}\label{eq:cell-pairs}
 \#\{(i,j):r_{A_i-A_j}(t)>0\}
 \leq C D^\sigma.
\end{equation}
Indeed, on each connected component of
\[
 W_t\setminus\bigl(Z(P)\cup Z(P(\cdot-t))\bigr),
\]
the points $x$ and $x-t$ remain in fixed partition cells.

By Cauchy--Schwarz and \eqref{eq:cell-pairs},
\begin{equation}\label{eq:first-crossing}
 r_Y(t)^2
 \leq C D^\sigma\sum_{i,j}r_{A_i-A_j}(t)^2,
 \qquad t\notin\Bad(W).
\end{equation}
Summing and using \cref{lem:exceptional-energy},
\begin{equation}\label{eq:first-crossing-summed}
 E(Y)
 \leq 2\lambda_W(A)\card A^2
 +CD^\sigma\sum_{i,j}\sum_t r_{A_i-A_j}(t)^2.
\end{equation}

By \eqref{eq:mixed-identity},
\begin{equation}\label{eq:crossing-rearrange}
 \sum_{i,j}\sum_t r_{A_i-A_j}(t)^2
 =\sum_u\left(\sum_i r_{A_i-A_i}(u)\right)^2.
\end{equation}
For $u\notin\Bad(W)$, the argument giving \eqref{eq:cell-pairs}, now with $i=j$, yields
\[
 \#\{i:r_{A_i-A_i}(u)>0\}\leq CD^\sigma.
\]
Hence
\begin{equation}\label{eq:second-crossing-good}
 \left(\sum_i r_{A_i-A_i}(u)\right)^2
 \leq CD^\sigma\sum_i r_{A_i-A_i}(u)^2.
\end{equation}
For $u\in\Bad(W)$, the left side is bounded by $r_A(u)^2$, and \cref{lem:exceptional-energy} applies.  Thus
\begin{equation}\label{eq:second-crossing-summed}
 \sum_{i,j}\sum_t r_{A_i-A_j}(t)^2
 \leq 2\lambda_W(A)\card A^2
 +CD^\sigma\sum_iE(A_i).
\end{equation}
Substitute \eqref{eq:second-crossing-summed} into \eqref{eq:first-crossing-summed} and absorb lower powers of $D$.
\end{proof}

\subsection{Proof of the flagged theorem}
\begin{proof}[Proof of \cref{thm:flagged-main}]
We prove a slightly stronger target-exponent statement by induction on the dimension.  Fix $a\in[2,3)$ and assume $\alpha(V)\leq a$.

We first choose a finite degree tower.  Starting with the top degree bound $\Delta_m=\Delta$, descend in dimension.  At dimension $k$, choose a partition degree $D_k$ sufficiently large in terms of $n,\Delta_k,a,\eps$ that the cellular factor in \eqref{eq:cell-contraction} below is at most $1/100$.  Then choose $\Delta_{k-1}$ large enough to dominate the degrees of all irreducible components of an intersection of a degree-$\Delta_k$ variety with a degree-$D_k$ hypersurface.  The process ends after at most $m$ steps.  Let
\[
 R=\max_k\Delta_k.
\]
All constants below are uniform for the finitely many dimensions and degree bounds in this tower.

Let $W$ be a node of the resulting induction, let $k=\dim W$, and let $A\subset X\cap W$ have size $M$.  We only need to treat nodes with $\alpha(W)\leq a$; nodes with larger intrinsic exponent will be bounded directly.  Put
\[
 \Lambda=\Lambda_{a,R}(X;V).
\]
Then $\lambda_W(A)\leq\Lambda$.

Apply polynomial partitioning to $A$ at degree $D=D_k$.  By \cref{prop:two-crossing},
\begin{equation}\label{eq:node-cell}
 E(Y)
 \leq C\left(D^{\sigma(W)}\Lambda M^2
 +D^{2\sigma(W)}\sum_iE(A_i)\right).
\end{equation}
The cells satisfy
\begin{equation}\label{eq:cell-size}
 \card{A_i}\leq C\frac{M}{D^k},
 \qquad
 \sum_i\card{A_i}\leq M.
\end{equation}
By strong induction on cardinality,
\[
 E(A_i)\leq C_k\Lambda^{3-a}\card{A_i}^{a+\eps}.
\]
Consequently,
\begin{align}
 D^{2\sigma(W)}\sum_iE(A_i)
 &\leq C C_k\Lambda^{3-a}D^{2\sigma(W)}
 \left(\frac{M}{D^k}\right)^{a+\eps-1}M\notag\\
 &=C C_k\Lambda^{3-a}M^{a+\eps}
 D^{2\sigma(W)-k(a+\eps-1)}.\label{eq:cell-contraction}
\end{align}
Since $a\geq\alpha(W)$,
\[
 a\geq1+\frac{2\sigma(W)}k,
\]
so the power of $D$ in \eqref{eq:cell-contraction} is at most $-k\eps$.  The choice of $D_k$ makes this contribution a fixed small fraction of the inductive target.

By \eqref{eq:lambda-interpolation},
\begin{equation}\label{eq:bad-term-absorb}
 D^{\sigma(W)}\Lambda M^2
 \leq D^{\sigma(W)}\Lambda^{3-a}M^a.
\end{equation}
For $M$ above a fixed threshold, the missing factor $M^\eps$ absorbs the constant in \eqref{eq:bad-term-absorb}; the remaining bounded range is included in the induction base.

It remains to control the wall
\[
 A_0=A\cap W\cap Z(P).
\]
By Bezout and irreducible decomposition, $A_0$ is contained in a union of $O_{n,\Delta_k,D_k}(1)$ admissible irreducible varieties $U_j$ of dimension strictly less than $k$ and degree at most $\Delta_{k-1}$.  Assign each wall point to one component and write $A_0=\bigsqcup_jB_j$.

If $\alpha(U_j)\leq a$, the induction on dimension gives
\[
 E(B_j)\leq C_{k-1}\Lambda^{3-a}\card{B_j}^{a+\eps}.
\]
If $\alpha(U_j)>a$, the definition of $\Lambda$ gives $\card{B_j}\leq\Lambda$, and the trivial estimate yields
\begin{equation}\label{eq:worse-wall}
 E(B_j)\leq\card{B_j}^3
 \leq\Lambda^{3-a}\card{B_j}^a
 \leq\Lambda^{3-a}\card{B_j}^{a+\eps}.
\end{equation}
Zero-dimensional components contain only $O_{n,R}(1)$ real points and are harmless.  Using \eqref{eq:union-fourth-root} and the bounded number of components,
\begin{equation}\label{eq:wall-bound}
 E(A_0)\leq C'_{k-1}\Lambda^{3-a}M^{a+\eps}.
\end{equation}

Finally, $A=Y\sqcup A_0$, and \eqref{eq:union-fourth-root} implies
\[
 E(A)\leq8\bigl(E(Y)+E(A_0)\bigr).
\]
Choose the dimension-$k$ induction constant larger than the fixed wall constant and the bounded-cardinality base constant.  The cellular term is contractive by \eqref{eq:cell-contraction}.  This completes both inductions and proves \eqref{eq:flagged-main-intro}.
\end{proof}

\subsection{Consequences and recovery of the surface theorem}
\begin{corollary}\label{cor:intrinsic-flag}
Let $V$ be admissible irreducible with $\alpha(V)<3$, and put $a=\alpha(V)$.  Then
\begin{equation}\label{eq:intrinsic-flag}
 E(X)\ll_{V,\eps}
 \Lambda_{a,R}(X;V)^{3-a}\card X^{a+\eps}.
\end{equation}
If the flag parameter is bounded independently of $X$, then
\[
 E(X)\ll_{V,\eps}\card X^{\alpha(V)+\eps}.
\]
\end{corollary}

\begin{corollary}\label{cor:recover-jw}
Let $V\subset\R^3$ be an admissible irreducible nonplanar algebraic surface of bounded degree.  Then the choice $a=2$ in \cref{thm:flagged-main} gives
\begin{equation}\label{eq:recover-jw}
 E(X)\ll_{\deg V,\eps}\Lambda_V(X)\card X^{2+\eps},
 \qquad X\subset V,
\end{equation}
where $\Lambda_V(X)=\max\{1,\sup_{\ell\subset V}\card{X\cap\ell}\}$ and the supremum is over affine lines contained in $V$.
\end{corollary}

\begin{proof}
We first note that $\dim\cD(V)=3$.  Since a translate of $V$ is contained in $V-V$, this dimension is at least two.  If it were equal to two, then for generic smooth $x,y\in V$ the image of the differential of the difference map would be $T_xV+T_yV$ and would have dimension at most two.  Hence $T_xV=T_yV$ for generic pairs.  The tangent plane is therefore constant on a dense open subset; every linear functional annihilating that plane is constant on $V$, so irreducibility places $V$ in an affine plane, contrary to the hypothesis.  Hence $\sigma(V)=1$ and $\alpha(V)=2$.

If $t\notin H_V$, the two irreducible hypersurfaces $V_{\C}$ and $V_{\C}+t$ have no common two-dimensional component, so there is no non-stabilizer excess fiber.  The stabilizer has dimension at most one, since a two-dimensional stabilizer would make $V$ an affine plane.  Its nontrivial cosets are affine lines contained in $V$.

For a positive-dimensional proper subvariety of $V$, the only case with intrinsic exponent larger than two is an affine line.  Indeed, if an irreducible curve had one-dimensional difference variety, the same tangent-space argument would force its tangent line to be constant, hence the curve would be affine.  A nonlinear irreducible curve therefore has two-dimensional difference variety, zero-dimensional generic difference fibers, and no nonzero translation stabilizer.  Bezout bounds all its nonzero translate intersections by a degree-dependent constant.  Thus the flag parameter is comparable, with degree-dependent constants, to $\Lambda_V(X)$.  Apply \cref{thm:flagged-main}.
\end{proof}

\begin{remark}
The proof above is not intended to replace the sharper and more direct wall analysis in \cite{JingWu}.  Its purpose is to show that their line parameter is the first member of a general translation--partition flag hierarchy.
\end{remark}

\section{A codimension-two quadratic threefold}\label{sec:quadratic-threefold}

\subsection{Normal form and difference fibers}
Let
\[
 \Sigma(Q_1,Q_2)=\{(u,Q_1(u),Q_2(u)):u\in\R^3\}\subset\R^5.
\]
Invertible linear transformations preserve additive relations and energy.  If $Q_1$ is positive definite and the generalized spectrum is simple, an invertible linear change in $u$ and an invertible linear change in the last two target coordinates reduce the problem to
\begin{equation}\label{eq:Phi-normal}
 \Phi(u)=\bigl(u,|u|^2,u^{\mathsf T}\Lambda u\bigr),
 \qquad
 \Lambda=\operatorname{diag}(\lambda_1,\lambda_2,\lambda_3),
 \quad \lambda_i\neq\lambda_j.
\end{equation}
We work in this normal form.

\begin{lemma}\label{lem:no-lines}
The variety $\Sigma=\Phi(\R^3)$ contains no nontrivial affine line.  More precisely, every affine line in $\R^5$ meets $\Sigma$ in at most two points unless it is contained in $\Sigma$, and the latter alternative does not occur.
\end{lemma}

\begin{proof}
If a line is contained in $\Sigma$, its projection onto the first three coordinates is either a point or a line $u_0+sv$.  The first case gives a constant line.  In the second case, the coordinate $|u_0+sv|^2$ has quadratic coefficient $|v|^2$, which cannot vanish for $v\neq0$.  Thus no nontrivial line is contained in $\Sigma$.  The intersection bound follows from Bezout, or directly from the quadratic coordinate.
\end{proof}

Let $A\subset\R^3$ be finite and put $X=\Phi(A)$.  Fix a difference
\[
 t=\Phi(u)-\Phi(v)=(h,\tau_1,\tau_2),
 \qquad h=u-v.
\]
Then $u$ satisfies
\begin{align}
 2u\cdot h-|h|^2&=\tau_1,\label{eq:fiber-one}\\
 2u\cdot\Lambda h-h\cdot\Lambda h&=\tau_2.\label{eq:fiber-two}
\end{align}
If $h$ and $\Lambda h$ are linearly independent, these equations define an affine line in parameter space.  Since the eigenvalues are distinct, dependence occurs exactly when
\begin{equation}\label{eq:rankdrop-directions}
 h\in\R e_1\cup\R e_2\cup\R e_3.
\end{equation}

\subsection{The rank-drop energy}
Let $\cR$ denote the set of differences whose parameter displacement belongs to the union in \eqref{eq:rankdrop-directions}.

\begin{lemma}\label{lem:gram-rankdrop}
For every finite $A\subset\R^3$ and $X=\Phi(A)$,
\begin{equation}\label{eq:gram-rankdrop}
 \sum_{t\in\cR}r_X(t)^2\leq 7\card X^2.
\end{equation}
\end{lemma}

\begin{proof}
We treat the direction $e_1$.  Let $R$ be the finite set of first coordinates occurring in $A$, let $Z$ be the finite set of pairs $(u_2,u_3)$ occurring in $A$, and define the incidence matrix
\[
 M_{r,z}=\1_A(r,z),
 \qquad r\in R,\ z\in Z.
\]
For $r\neq s$, the full difference associated with a pair $(r,z),(s,z)$ determines the ordered pair $(r,s)$, since
\[
 h=r-s,
 \qquad
 r^2-s^2=h(r+s).
\]
Its multiplicity is therefore
\[
 G_{r,s}=\sum_zM_{r,z}M_{s,z},
 \qquad G=MM^{\mathsf T}.
\]
Hence the nonzero $e_1$-direction contribution is at most
\[
 \sum_{r\neq s}G_{r,s}^2
 \leq\operatorname{tr}(G^2)
 \leq\operatorname{tr}(G)^2
 =\card A^2,
\]
because $G$ is positive semidefinite and $\operatorname{tr}(G)=\card A$.  The zero difference contributes $\card A^2$.  Thus the $e_1$ family contributes at most $2\card A^2$.  The same argument applies to $e_2$ and $e_3$.  Counting the zero difference only once gives \eqref{eq:gram-rankdrop}.
\end{proof}

\subsection{A finite Freiman projection for wall surfaces}
We need a version of the Jing--Wu theorem for a finite subset of a surface in an arbitrary ambient space.

\begin{lemma}\label{lem:freiman-projection}
Let $S\subset\R^M$ be an admissible irreducible algebraic surface of degree at most $B$ that contains no affine line, and let $A\subset S$ be finite.  There is a linear map
\[
 L:\R^M\longrightarrow\R^3
\]
such that:
\begin{enumerate}[label=(\roman*)]
 \item $L$ is a Freiman isomorphism of order two on $A$;
 \item the Zariski closure $S'=\overline{L(S)}^{\,\mathrm{Zar}}$ is an irreducible nonplanar algebraic surface of degree $O_B(1)$;
 \item every affine line in $S'$ contains at most $B$ points of $L(A)$.
\end{enumerate}
Consequently,
\begin{equation}\label{eq:surface-any-ambient}
 E(A)\ll_{B,\eps}\card A^{2+\eps}.
\end{equation}
\end{lemma}

\begin{proof}
Choose $L$ generically.  The failure of the Freiman property means that $\ker L$ contains a nonzero vector from the finite set
\[
 (A+A-A-A)\setminus\{0\},
\]
so this is avoided by a generic choice.  We also require that every noncollinear triple in $A$ remain noncollinear after projection; each failure is a proper algebraic condition on $L$.

The linear maps for which the image dimension drops, the affine span of the image is planar, or the generic projection degree/degree bound fails form a proper algebraic exceptional set.  Thus a generic linear projection of an irreducible surface has two-dimensional irreducible image closure and degree bounded in terms of $B$; this is the standard generic-projection theorem, together with the degree bound for a linear image.  Since $S$ contains no affine plane, its affine span has dimension at least three, and $L$ may be chosen so that the image is not planar.

Suppose an affine line $\ell\subset S'$ contains more than $B$ points of $L(A)$.  By preservation of noncollinearity, their preimages lie on one affine line in $\R^M$.  That line meets $S$ in more than $B$ points, so Bezout forces it to be contained in $S$, a contradiction.  Thus the line concentration of $L(A)$ is at most $B$.  Since $L$ preserves all additive quadruples, $E(A)=E(L(A))$, and \eqref{eq:surface-any-ambient} follows from \eqref{eq:jw-intro}.
\end{proof}

For curves, an even simpler observation suffices.

\begin{lemma}\label{lem:curve-energy}
Let $C\subset\R^M$ be an admissible irreducible algebraic curve of degree at most $B$ that is not an affine line.  Then every finite $A\subset C$ satisfies
\begin{equation}\label{eq:curve-energy}
 E(A)\ll_B\card A^2.
\end{equation}
\end{lemma}

\begin{proof}
If $t\neq0$ and $C=C+t$, then the translation stabilizer contains a line, forcing $C$ itself to be an affine line.  Hence $C$ and $C+t$ are distinct for every $t\neq0$, and Bezout gives
\[
 r_A(t)\leq\card{C\cap(C+t)}\ll_B1.
\]
The zero difference contributes $\card A^2$, and
\[
 \sum_{t\neq0}r_A(t)^2\ll_B\sum_t r_A(t)=O_B(\card A^2).
\]
\end{proof}

\subsection{Polynomial partitioning in parameter space}
Apply the three-dimensional polynomial partitioning theorem to $A\subset\R^3$ at degree $D$.  Let
\[
 A_0=A\cap Z(q),
 \qquad
 A\setminus A_0=\bigsqcup_i A_i,
 \qquad
 X_i=\Phi(A_i),
 \qquad
 Y=\bigsqcup_iX_i.
\]
Then
\begin{equation}\label{eq:R3-cells}
 \card{A_i}\ll\frac{\card A}{D^3},
 \qquad
 \#\{i:A_i\neq\varnothing\}\ll D^3.
\end{equation}

\begin{proposition}\label{prop:threefold-cell}
With the notation above,
\begin{equation}\label{eq:threefold-cell}
 E(Y)\ll D\card X^2+D^2\sum_iE(X_i).
\end{equation}
\end{proposition}

\begin{proof}
For a difference $t\notin\cR$, the solution set of \eqref{eq:fiber-one}--\eqref{eq:fiber-two} is an affine line $L_t$ in parameter space.  The restrictions of $q(u)$ and $q(u-h)$ to this line have degree at most $D$.  If either restriction vanishes identically, there is no contribution from $Y$ on the corresponding side.  Otherwise their zero sets cut $L_t$ into $O(D)$ intervals.  On each interval, $u$ and $u-h$ remain in fixed partition cells.  Hence
\begin{equation}\label{eq:threefold-first-count}
 \#\{(i,j):r_{X_i-X_j}(t)>0\}\ll D,
 \qquad t\notin\cR.
\end{equation}
Cauchy--Schwarz and \cref{lem:gram-rankdrop} give
\begin{equation}\label{eq:threefold-first-cross}
 E(Y)\ll\card X^2
 +D\sum_{i,j}\sum_t r_{X_i-X_j}(t)^2.
\end{equation}
By \eqref{eq:mixed-identity},
\[
 \sum_{i,j}\sum_t r_{X_i-X_j}(t)^2
 =\sum_s\left(\sum_i r_{X_i-X_i}(s)\right)^2.
\]
For $s\notin\cR$, \eqref{eq:threefold-first-count} with $i=j$ gives at most $O(D)$ nonzero summands.  The contribution from $s\in\cR$ is bounded by \cref{lem:gram-rankdrop}.  Therefore
\begin{equation}\label{eq:threefold-second-cross}
 \sum_{i,j}\sum_t r_{X_i-X_j}(t)^2
 \ll\card X^2+D\sum_iE(X_i).
\end{equation}
Substitution into \eqref{eq:threefold-first-cross} proves the proposition.
\end{proof}

\subsection{The wall and the proof of the theorem}
Let
\[
 W=\Phi(A_0).
\]
Factor $q$ and decompose $A_0$ among the real irreducible components of its zero set.  Under the polynomial embedding $\Phi$, each positive-dimensional component becomes an algebraic subvariety of $\Sigma$ of dimension at most two and degree $O_D(1)$.  By \cref{lem:no-lines}, none contains an affine line.  Two-dimensional components satisfy \eqref{eq:surface-any-ambient}; one-dimensional components satisfy \eqref{eq:curve-energy}; zero-dimensional components contain $O_D(1)$ points.  The fourth-root union inequality therefore gives, for every $\eta>0$,
\begin{equation}\label{eq:threefold-wall}
 E(W)\ll_{D,\eta}\card W^{2+\eta}.
\end{equation}

\begin{proof}[Proof of \cref{thm:quadratic-threefold-intro}]
We argue by strong induction on $N=\card X$.  Take $\eta=\eps/2$ and fix a sufficiently large constant degree $D=D(\eps)$.  By \eqref{eq:union-fourth-root}, \cref{prop:threefold-cell}, and \eqref{eq:threefold-wall},
\begin{equation}\label{eq:threefold-recursion}
 E(X)
 \ll D N^2+D^2\sum_iE(X_i)+C_{D,\eps}N^{2+\eps/2}.
\end{equation}
Using the inductive hypothesis and \eqref{eq:R3-cells},
\begin{align}
 D^2\sum_iE(X_i)
 &\ll C_\eps D^2
 \left(\frac{N}{D^3}\right)^{1+\eps}
 \sum_i\card{X_i}\notag\\
 &\ll C_\eps D^{-1-3\eps}N^{2+\eps}.
\end{align}
Choose $D$ so that the coefficient of $C_\eps N^{2+\eps}$ is less than $1/4$.  The first and third terms in \eqref{eq:threefold-recursion} are absorbed for $N$ above a fixed threshold, and the remaining finite range is included in the induction constant.  This proves \eqref{eq:quadratic-threefold-intro}.

The diagonal solutions give $E(X)\geq2N^2-N$, so the exponent is sharp.
\end{proof}

\begin{remark}\label{rem:repeated-spectrum}
If $Q_2=\lambda Q_1$, then $\Sigma(Q_1,Q_2)$ lies in an affine hyperplane and is linearly equivalent to a three-dimensional paraboloid in $\R^4$.  Lattice boxes on that paraboloid have energy of order $N^{7/3}$; see the discussion surrounding the paraboloid threshold in \cite{Mudgal}.  Thus a rank condition on the quadratic pencil is genuinely necessary for a universal $N^{2+\eps}$ theorem.  The intermediate multiplicity pattern $2+1$ should admit a refinement using lower-dimensional paraboloid fibers, but it is not treated here.
\end{remark}

\section{Weighted and off-diagonal Young inequalities}\label{sec:weighted-young}

\subsection{From hereditary energy to weighted \texorpdfstring{$L^4$}{L4}}
We first record an abstract extension principle.

\begin{proposition}\label{prop:restricted-to-strong}
Let $X$ be finite, let $2\leq a<3$, and let $K\geq1$.  Suppose that for every $\eta>0$ there is $C_\eta$ such that
\begin{equation}\label{eq:hereditary-energy}
 E(A)\leq C_\eta K^{3-a}\card A^{a+\eta}
 \qquad(A\subset X).
\end{equation}
Then, for every $\eps>0$, there is $\eta=\eta(a,\eps)>0$ such that every $f$ supported on $X$ satisfies
\begin{equation}\label{eq:weighted-L4}
 \norm[L^4]{\widehat f}
 \ll_{a,\eps} C_\eta^{1/4}K^{(3-a)/4}\card X^\eps
 \norm[\ell^{4/a}]{f}.
\end{equation}
\end{proposition}

\begin{proof}
Choose $\eta=\eta(a,\eps)>0$ sufficiently small and put
\[
 \frac1{p_\eta}=\frac{a+\eta}{4}.
\]
The indicator estimate \eqref{eq:hereditary-energy} says that the Fourier extension operator is of restricted type $(p_\eta,4)$ with norm $O(C_\eta^{1/4}K^{(3-a)/4})$.  The standard restricted-type extension (equivalently, a dyadic decomposition of the coefficients) gives
\begin{equation}\label{eq:lorentz-step}
 \norm[L^4]{\widehat f}
 \ll C_\eta^{1/4}K^{(3-a)/4}\norm[\ell^{p_\eta,1}]{f}.
\end{equation}
For a sequence supported on at most $N=\card X$ points,
\[
 \norm[\ell^{p_\eta,1}]{f}
 \ll_{p_\eta}(\log(2N))^{1-1/p_\eta}\norm[\ell^{p_\eta}]{f}.
\]
Since $p_\eta<4/a$,
\[
 \norm[\ell^{p_\eta}]{f}
 \leq N^{\eta/4}\norm[\ell^{4/a}]{f}.
\]
Choose $\eta$ sufficiently small in terms of $\eps$ and absorb the logarithm into $N^\eps$.  This proves the claim.
\end{proof}

Applying \cref{thm:flagged-main} to every subset of $X$ is legitimate because the flag parameter is monotone.  We therefore obtain the following.

\begin{corollary}\label{cor:weighted-flag}
In the setting of \cref{thm:flagged-main}, let $\Lambda=\Lambda_{a,R}(X;V)$.  Then
\begin{equation}\label{eq:weighted-flag}
 \norm[L^4]{\widehat f}
 \ll_{V,a,\eps}
 \Lambda^{(3-a)/4}\card X^\eps
 \norm[\ell^{4/a}]{f},
 \qquad \supp f\subset X.
\end{equation}
\end{corollary}

For the quadratic threefold, \cref{thm:quadratic-threefold-intro} gives
\begin{equation}\label{eq:weighted-threefold}
 \norm[L^4]{\widehat f}
 \ll_\eps\card X^\eps\norm[\ell^2]{f}.
\end{equation}

\subsection{The off-diagonal interpolation polygon}
Let $X,Y$ satisfy hereditary estimates with exponents $a_X,a_Y$ and parameters $K_X,K_Y$.  Put
\[
 p_X=\frac4{a_X},
 \qquad
 p_Y=\frac4{a_Y}.
\]
The geometric endpoint obtained from \eqref{eq:bilinear-correlation} and \cref{prop:restricted-to-strong} is
\begin{equation}\label{eq:geometric-bilinear-endpoint}
 \norm[\ell^2]{f*g}
 \ll (K_X^{3-a_X}K_Y^{3-a_Y})^{1/4}
 (\card X\card Y)^\eps
 \norm[\ell^{p_X}]{f}\norm[\ell^{p_Y}]{g}.
\end{equation}
The classical endpoints are
\begin{equation}\label{eq:classical-bilinear-endpoints}
 \norm[\ell^2]{f*g}\leq\norm[\ell^1]{f}\norm[\ell^2]{g},
 \qquad
 \norm[\ell^2]{f*g}\leq\norm[\ell^2]{f}\norm[\ell^1]{g}.
\end{equation}

\begin{theorem}\label{thm:general-offdiag}
Let $X,Y$ be finite sets.  Suppose that, for every $\eta>0$ and all subsets $A\subset X$, $B\subset Y$,
\[
 E(A)\ll_\eta K_X^{3-a_X}\card A^{a_X+\eta},
 \qquad
 E(B)\ll_\eta K_Y^{3-a_Y}\card B^{a_Y+\eta},
\]
where $2\leq a_X,a_Y<3$ and $K_X,K_Y\geq1$.  Let $\theta_0,\theta_1,\theta_2\geq0$ with $\theta_0+\theta_1+\theta_2=1$, and define
\begin{align}
 \frac1p&=\theta_0\frac{a_X}{4}+\theta_1+\frac{\theta_2}{2},\label{eq:p-barycentric}\\
 \frac1q&=\theta_0\frac{a_Y}{4}+\frac{\theta_1}{2}+\theta_2.\label{eq:q-barycentric}
\end{align}
Then
\begin{equation}\label{eq:general-offdiag}
 \norm[\ell^2]{f*g}
 \ll_\eps
 K_X^{\theta_0(3-a_X)/4}
 K_Y^{\theta_0(3-a_Y)/4}
 (\card X\card Y)^\eps
 \norm[\ell^p]{f}\norm[\ell^q]{g}.
\end{equation}
The same estimate holds after decreasing either $p$ or $q$.
\end{theorem}

\begin{proof}
Apply bilinear complex interpolation to \eqref{eq:geometric-bilinear-endpoint} and the two estimates in \eqref{eq:classical-bilinear-endpoints}; see, for example, \cite{BerghLofstrom,Grafakos}.  The final assertion follows from monotonicity of counting-measure norms.
\end{proof}

When $a_X=a_Y=2$, the interpolation polygon and its upward monotone closure admit a simple description.

\begin{corollary}\label{cor:flagged-offdiag}
Let $V,W\subset\R^n$ be admissible irreducible bounded-degree varieties.  Choose target exponents $a_V,a_W\in[2,3)$ with $\alpha(V)\leq a_V$ and $\alpha(W)\leq a_W$, and let $R_V,R_W$ be the degree cutoffs supplied by \cref{thm:flagged-main}.  For finite $X\subset V$, $Y\subset W$, put
\[
 \Lambda_X=\Lambda_{a_V,R_V}(X;V),
 \qquad
 \Lambda_Y=\Lambda_{a_W,R_W}(Y;W).
\]
If $\theta_0,\theta_1,\theta_2$ and $p,q$ satisfy \eqref{eq:p-barycentric}--\eqref{eq:q-barycentric} with $a_X=a_V$ and $a_Y=a_W$, then
\begin{equation}\label{eq:flagged-offdiag}
 \norm[\ell^2]{f*g}
 \ll_{n,\deg V,\deg W,p,q,\eps}
 \Lambda_X^{\theta_0(3-a_V)/4}
 \Lambda_Y^{\theta_0(3-a_W)/4}
 (\card X\card Y)^\eps
 \norm[\ell^p]{f}\norm[\ell^q]{g}.
\end{equation}
The estimate remains valid after decreasing either input exponent.
\end{corollary}

\begin{proof}
Apply \cref{cor:weighted-flag,thm:general-offdiag} and use the monotonicity of the flag parameters on subsets.
\end{proof}

\begin{corollary}\label{cor:alpha2-offdiag}
Assume the hypotheses of \cref{thm:offdiag-alpha2-intro}.  Then \eqref{eq:offdiag-alpha2-intro} holds.
\end{corollary}

\begin{proof}
For
\[
 1\leq\frac1p+\frac1q\leq\frac32,
\]
set
\[
 \theta_0=3-2\left(\frac1p+\frac1q\right).
\]
The remaining barycentric mass can be split between the two classical endpoints because $1/2\leq1/p,1/q\leq1$.  If $1/p+1/q\geq3/2$, use the classical Young line and norm monotonicity, so $\theta_0=0$.
\end{proof}

\subsection{Applications and sharpness}
\begin{corollary}\label{cor:surface-offdiag}
Let $X\subset V$ and $Y\subset W$, where $V,W\subset\R^3$ are admissible irreducible nonplanar algebraic surfaces of bounded degree.  For $1\leq p,q\leq2$ with $1/p+1/q\geq1$, let
\[
 \theta=\max\left\{0,3-2\left(\frac1p+\frac1q\right)\right\}.
\]
Then
\begin{equation}\label{eq:surface-offdiag}
 \norm[\ell^2]{f*g}
 \ll_{\deg V,\deg W,p,q,\eps}
 \bigl(\Lambda_V(X)\Lambda_W(Y)\bigr)^{\theta/4}
 (\card X\card Y)^\eps
 \norm[\ell^p]{f}\norm[\ell^q]{g}.
\end{equation}
\end{corollary}

\begin{corollary}\label{cor:threefold-offdiag}
Let $X,Y$ lie on quadratic threefolds satisfying the hypotheses of \cref{thm:quadratic-threefold-intro}.  Then, for every $1\leq p,q\leq2$ with $1/p+1/q\geq1$,
\begin{equation}\label{eq:threefold-offdiag}
 \norm[\ell^2]{f*g}
 \ll_{p,q,\eps}(\card X\card Y)^\eps
 \norm[\ell^p]{f}\norm[\ell^q]{g}.
\end{equation}
\end{corollary}

The exponent region in \cref{cor:threefold-offdiag} is optimal up to the subpolynomial loss.  Indeed, taking $g$ to be a point mass forces $p\leq2$, and symmetrically $q\leq2$.  Taking $f=g=\1_X$ and using the diagonal lower bound $E(X)\gtrsim\card X^2$ forces
\[
 \frac1p+\frac1q\geq1.
\]

The concentration power in \cref{cor:surface-offdiag} is also sharp in the joint sense.  Let $X=Y$ be an arithmetic progression of length $N$ on a line contained in the surface.  Then
\[
 \norm[\ell^2]{\1_X*\1_X}\asymp N^{3/2},
 \qquad
 \Lambda_V(X)=N.
\]
When $1\leq1/p+1/q\leq3/2$, the right side of \eqref{eq:surface-offdiag} has $N$-exponent
\[
 \frac1p+\frac1q+\frac\theta2=\frac32.
\]
Thus the displayed power of the product concentration parameter cannot be uniformly reduced.

\subsection{Critical exponents from dimension}
If $V^m\subset\R^{m+c}$ has maximal difference dimension, then \eqref{eq:mc-intro} and \cref{cor:weighted-flag} give the critical symmetric input exponent
\begin{equation}\label{eq:critical-p}
 p_c(m,c)=\frac4{\alpha_{m,c}}
 =\begin{cases}
 2,&m\leq2c,\\[1mm]
 \dfrac{4m}{3m-2c},&m>2c.
 \end{cases}
\end{equation}
Thus codimension two gives $p_c=2$ for $m=3,4$, while $p_c=20/11$ for $m=5$.  At exponents $a>2$, \cref{thm:general-offdiag} supplies the interpolation polygon generated by the symmetric geometric endpoint and the two classical Young endpoints.  Determining a larger genuinely asymmetric sharp region would require new mixed-energy information and is not a consequence of self-energy alone.

\section{Turning complexity and three-fold energy on curves}\label{sec:turning}

\subsection{Signed-convex complexity}
Let $P\subset\R^2$ be finite and let $\ell$ be a linear functional that is injective on $P$.  Choose a linear functional $m$ linearly independent of $\ell$ and enumerate
\[
 P=\{p_1,\dots,p_N\},
 \qquad
 \ell(p_1)<\cdots<\ell(p_N).
\]
Put
\begin{equation}\label{eq:adjacent-slopes}
 s_i=\frac{m(p_{i+1})-m(p_i)}{\ell(p_{i+1})-\ell(p_i)},
 \qquad 1\leq i<N.
\end{equation}
A consecutive point block $\{p_a,\dots,p_b\}$ is \emph{signed-convex} if $s_a,\dots,s_{b-1}$ is strictly increasing or strictly decreasing.  Let $\kappa_\ell(P)$ be the least number of consecutive signed-convex blocks in a partition of $P$.

If $m'$ is another complementary coordinate, then $m'=cm+d\ell$ with $c\neq0$, so the new slopes are $cs_i+d$.  Thus $\kappa_\ell(P)$ is independent of $m$.  One may define the affine-invariant complexity
\[
 \kappa(P)=\min_\ell\kappa_\ell(P),
\]
where the minimum is over projections injective on $P$.

\begin{lemma}\label{lem:convex-interpolation}
If the adjacent slopes of a finite graph data set
\[
 x_1<\cdots<x_M,
 \qquad (x_i,y_i)\in\R^2
\]
are strictly increasing, then there is a strictly convex function $f:\R\to\R$ with $f(x_i)=y_i$ for every $i$.  If the slopes are strictly decreasing, there is a strictly concave interpolant.
\end{lemma}

\begin{proof}
We treat the increasing case.  Let
\[
 \sigma_i=\frac{y_{i+1}-y_i}{x_{i+1}-x_i}.
\]
Choose derivatives $d_i$ so that
\[
 d_1<\sigma_1,
 \qquad
 \sigma_{i-1}<d_i<\sigma_i\quad(2\leq i\leq M-1),
 \qquad
 d_M>\sigma_{M-1}.
\]
On each interval $[x_i,x_{i+1}]$, choose a continuous strictly increasing function $g_i$ with endpoint values $d_i,d_{i+1}$ and average $\sigma_i$.  Such a function is obtained by moving the transition point in a monotone piecewise-linear profile; it may be smoothed without changing the endpoint values or the integral.  Define
\[
 f(x)=y_i+\int_{x_i}^xg_i(t)\,dt
 \qquad(x\in[x_i,x_{i+1}]).
\]
Then $f$ is $C^1$, has strictly increasing derivative, and interpolates the data.  Extend the derivative monotonically outside $[x_1,x_M]$.  The decreasing case follows by replacing $y$ with $-y$.
\end{proof}

\subsection{A strip projection estimate}
\begin{lemma}\label{lem:L6-strips}
Let $g_1,\dots,g_K\in L^6(\R^2)$ have Fourier supports contained in pairwise disjoint strips of the form
\[
 \{\xi:a_j<\ell(\xi)<b_j\},
\]
with disjoint intervals $(a_j,b_j)$.  For every $\delta>0$,
\begin{equation}\label{eq:L6-strips}
 \norm[L^6]{\sum_{j=1}^K g_j}^6
 \ll_\delta K^{4+\delta}\sum_{j=1}^K\norm[L^6]{g_j}^6.
\end{equation}
\end{lemma}

\begin{proof}
After a linear change of coordinates, take $\ell(\xi)=\xi_1$.  Let $P_j$ be the Fourier projection to the $j$th strip and define
\[
 T((h_j)_j)=\sum_jP_jh_j.
\]
Orthogonality gives
\[
 T:\ell^2(L^2)\longrightarrow L^2
\]
with norm one.  For every $r>6$, interval projections in one Fourier coordinate are uniformly bounded on $L^r$ by the Hilbert-transform theorem, and therefore
\[
 \norm[L^r]{T(h_j)}
 \ll_r\sum_j\norm[L^r]{h_j}
 \leq K^{1-1/r}\left(\sum_j\norm[L^r]{h_j}^r\right)^{1/r}.
\]
Interpolate these vector-valued estimates at the parameter $\theta$ determined by
\[
 \frac16=\frac{1-\theta}{2}+\frac\theta r.
\]
This gives
\[
 \norm[L^6]{T(h_j)}
 \ll_r K^{\theta(1-1/r)}
 \left(\sum_j\norm[L^6]{h_j}^6\right)^{1/6}.
\]
A calculation yields
\[
 6\theta\left(1-\frac1r\right)
 =\frac{4(r-1)}{r-2}
 =4+\frac4{r-2}.
\]
Choose $r$ so that $4/(r-2)<\delta$.
\end{proof}

\subsection{Admissible thickening}
For a finite $P\subset\R^2$, let
\[
 \cS_3(P)=\{p_1+p_2+p_3-p_4-p_5-p_6:p_i\in P\}.
\]
Choose a Schwartz function $\phi$ whose Fourier transform is supported in a sufficiently small ball and normalize $\int|\phi|^6=1$.  If the support is small enough that
\[
 \supp\widehat{|\phi|^6}\cap\bigl(\cS_3(P)\setminus\{0\}\bigr)=\varnothing,
\]
then
\begin{equation}\label{eq:admissible-thickening}
 \norm[L^6]{\phi(x)\sum_{p\in P}e^{2\pi ip\cdot x}}^6=J_3(P).
\end{equation}
This is the admissible thickening used in \cite{CDW}.  Since $P$ is finite, the Fourier support of $\phi$ may simultaneously be chosen small enough that different consecutive point blocks occupy disjoint $\ell$-strips.

\subsection{Proof of the complexity theorem}
\begin{proof}[Proof of \cref{thm:turning-intro}]
Fix a partition into $K=\kappa_\ell(P)$ consecutive signed-convex blocks
\[
 P=P_1\sqcup\cdots\sqcup P_K.
\]
Write $n_j=\card{P_j}$.  The invertible coordinate map $p\mapsto(\ell(p),m(p))$ preserves all additive relations.  By \cref{lem:convex-interpolation} and the theorem of Cushman--Demeter--Wu \cite{CDW}, for every $\eta>0$,
\begin{equation}\label{eq:leaf-CDW}
 J_3(P_j)\ll_\eta n_j^{3+\eta}.
\end{equation}
For a common admissible thickening, define
\[
 F_j(x)=\phi(x)\sum_{p\in P_j}e^{2\pi ip\cdot x}.
\]
Their Fourier supports lie in pairwise disjoint strips, and
\[
 J_3(P)=\norm[L^6]{\sum_jF_j}^6,
 \qquad
 \norm[L^6]{F_j}^6=J_3(P_j).
\]

The block sizes may be uneven.  Group them dyadically: for $r\geq0$, let
\[
 \cJ_r=\{j:2^r\leq n_j<2^{r+1}\},
 \qquad
 K_r=\card{\cJ_r},
 \qquad
 N_r=\sum_{j\in\cJ_r}n_j,
\]
and put $F^{(r)}=\sum_{j\in\cJ_r}F_j$.  By \cref{lem:L6-strips} and \eqref{eq:leaf-CDW},
\begin{align}
 \norm[L^6]{F^{(r)}}^6
 &\ll_\eta K_r^{4+\eta}\sum_{j\in\cJ_r}n_j^{3+\eta}\notag\\
 &\ll_\eta K_r^{5+\eta}2^{r(3+\eta)}
 \asymp K_r^2N_r^{3+\eta}.
\end{align}
Hence
\[
 \norm[L^6]{F^{(r)}}
 \ll_\eta K^{1/3}N_r^{1/2+\eta/6}.
\]
There are at most $1+\log_2N$ nonempty classes.  Minkowski and H\"older in the class index give
\begin{align}
 J_3(P)^{1/6}
 &\leq\sum_r\norm[L^6]{F^{(r)}}\notag\\
 &\ll_\eta K^{1/3}
 (1+\log N)^{1/2-\eta/6}N^{1/2+\eta/6}.
\end{align}
Taking sixth powers and choosing $\eta<\eps$ proves
\[
 J_3(P)\ll_\eps K^2N^{3+\eps}.
\]
\end{proof}

\subsection{Sharp examples}
\begin{proposition}\label{prop:turning-sharp}
For every pair of integers $K,n\geq1$, one can find a set $P\subset\R^2$ with $\card P=Kn$ and an injective projection $\ell$ such that
\[
 \kappa_\ell(P)\asymp K,
 \qquad
 J_3(P)\asymp K^2\card P^3.
\]
\end{proposition}

\begin{proof}
Set $N=Kn$.  Fix an integer $B\geq4$ and put
\[
 a_i=B^i,
 \qquad 1\leq i\leq n.
\]
Choose $H>3n$ and $C=a_n-a_1$, and define
\begin{equation}\label{eq:packet-example}
 P_{K,n}=\{(Hj+i,Cj+a_i):0\leq j<K,\ 1\leq i\leq n\}.
\end{equation}
Within each packet, the adjacent slopes are
\[
 a_{i+1}-a_i=(B-1)B^i
\]
and are strictly increasing.  The slope joining two consecutive packets is zero.  Thus the slope sequence has $K-1$ strict resets and
\[
 \kappa_\ell(P_{K,n})\asymp K
\]
for the first-coordinate projection.

Suppose two ordered triples in $P_{K,n}$ have the same sum.  The first coordinate gives
\[
 H\Delta J+\Delta I=0,
 \qquad \card{\Delta I}<3n<H,
\]
so $\Delta J=\Delta I=0$.  The second coordinate then gives $\Delta A=0$.  Consequently the equality separates into an equality of three-term sums in $\{0,\dots,K-1\}$ and an equality of three-term vector sums in
\[
 B_n=\{(i,B^i):1\leq i\leq n\}.
\]
Because $B\geq4$, a sum of three powers $B^i$ uniquely determines the multiset of exponents.  Hence
\[
 J_3(B_n)\asymp n^3.
\]
Also
\[
 J_3(\{0,\dots,K-1\})\asymp K^5.
\]
Therefore
\[
 J_3(P_{K,n})\asymp K^5n^3=K^2N^3.
\]
\end{proof}

Choosing integer sequences with $K\asymp N^\beta$ gives the full sharp power spectrum
\[
 J_3(P)\asymp N^{3+2\beta},
 \qquad 0\leq\beta\leq1.
\]

\section{Higher-dimensional consequences, sharpness, and limitations}\label{sec:higher-dimensional}

\subsection{Maximal difference dimension}
If $V^m\subset\R^{m+c}$ has maximal difference dimension, then \eqref{eq:mc-intro} gives
\begin{equation}\label{eq:generic-exponent}
 \alpha_{m,c}=\max\left\{2,3-\frac{2c}{m}\right\}.
\end{equation}
Thus
\[
 \begin{array}{c|c|c}
 (m,c)&\sigma&\alpha_{m,c}\\ \hline
 (2,1)&1&2\\
 (3,2)&1&2\\
 (4,2)&2&2\\
 (3,1)&2&7/3\\
 (5,2)&3&11/5
 \end{array}
\]
The cases $m\leq2c$ are near-diagonal; $m=2c$ is critical for the $N^\eps$ iteration; and $m<2c$ is subcritical, as illustrated by the genuine $D^{-1}$ gain in \cref{thm:quadratic-threefold-intro}.

\subsection{Quadratic lower bounds}
Let $Q_1,\dots,Q_c$ be fixed integer quadratic forms on $\R^m$ and consider
\[
 V_Q=\{(u,Q_1(u),\dots,Q_c(u)):u\in\R^m\}.
\]
For
\[
 X_M=\{(u,Q_1(u),\dots,Q_c(u)):u\in\{1,\dots,M\}^m\},
 \qquad N=M^m,
\]
one has
\[
 \card{X_M-X_M}\ll M^{m+2c}.
\]
Cauchy--Schwarz therefore gives
\begin{equation}\label{eq:quadratic-lower}
 E(X_M)\geq\frac{N^4}{\card{X_M-X_M}}
 \gg N^{3-2c/m}.
\end{equation}
For a generic choice of the forms $Q_j$, the difference map has maximal image dimension.  Together with the diagonal contribution, this shows that no uniform theorem for all bounded-degree varieties in the maximal-difference class can improve the exponent in \eqref{eq:generic-exponent}.  Special varieties can improve on the universal exponent: in four dimensions, sphere geometry breaks the $7/3$ threshold, while the paraboloid realizes it; see Mudgal \cite{Mudgal}.

\subsection{Sharpness of the flag power}
Suppose a variety $V_0$ admits finite sets $B$ with bounded base flag parameter and
\[
 E(B)\asymp\card B^a.
\]
Let $H$ be a one-dimensional affine factor and let $A\subset H$ be an arithmetic progression of length $L$.  On the cylinder $V=V_0\times H$, take $X=B\times A$.  Then
\[
 E(X)=E(B)E(A)\asymp\card B^aL^3.
\]
Writing $N=\card X=L\card B$, the flag parameter is comparable to $L$ under the stated bounded-base hypothesis, and one obtains
\begin{equation}\label{eq:flag-sharp}
 E(X)\asymp L^{3-a}N^a.
\end{equation}
Thus the power $3-a$ in \cref{thm:flagged-main} is forced whenever the base exponent is attained with bounded base flag complexity.  In particular, a generic finite subset of a parabola, tensored with an arithmetic progression, gives the sharp case $a=2$.

\subsection{What the flag theorem does not remove}
The flag theorem handles arbitrary bounded-degree varieties by charging the maximum occupancy of exceptional translate intersections.  This is sometimes too crude.  In the quadratic threefold, the rank-drop planes may contain many selected points, but \cref{lem:gram-rankdrop} shows that their total difference energy is only quadratic.  A sharper general theorem would replace the maximum flag occupancy by an aggregate quantity such as
\begin{equation}\label{eq:aggregate-exceptional}
 \mathfrak G_W(X)
 =\frac1{\card X^2}
 \sum_{\substack{t:\ \dim(W\cap(W+t))>\sigma(W)}}r_X(t)^2.
\end{equation}
Controlling \eqref{eq:aggregate-exceptional} requires additional geometry of the critical locus of the difference map.  For quadratic graphs it becomes a determinantal rank problem; for general varieties it is related to secant and contact defects and to translation rulings.  This is the principal obstruction to a parameter-free theorem for every bounded-degree variety.

A second open direction concerns the off-diagonal region when $\alpha(V)>2$.  Self-energy supplies a symmetric $\ell^{4/\alpha}\to L^4$ endpoint, and \cref{thm:general-offdiag} gives its interpolation polygon.  A larger sharp asymmetric region would require genuinely mixed geometric energy estimates rather than Cauchy--Schwarz between two self-energies.

\section{Statement on the use of AI}
AI systems were used extensively during exploratory work and preparation.  ChatGPT 5.6 Pro assisted with the discovery and checking of the difference-fiber recurrence, the flag formulation, the codimension-two quadratic model, the weighted interpolation argument, the turning-complexity examples, literature searches, exposition, and LaTeX preparation.  The author retains responsibility for independently verifying every argument and for all final mathematical and expository decisions.

\enlargethispage{4\baselineskip}

\end{document}